\documentclass[12pt]{article}
\usepackage{fullpage}
\usepackage{amsmath, amsthm, amssymb}
\numberwithin{equation}{section}
\newtheorem{theorem}{Theorem}
\newtheorem{lemma}{Lemma}

\begin{document}

\author{Ajai Choudhry}

\title{Two sets of integers such that all elements\\
 of the sumset of the two sets\\
are perfect squares}
\maketitle
\date{}

\abstract{This paper is concerned with the  problem of finding two  sets of  integers, $\{a_1, a_2, \ldots$,   $a_m\}$ and $\{b_1, b_2,  \ldots, b_n\}$,  such that all the  $mn$ sums $a_i+b_j, i=1, \ldots, m, j=1, \ldots, n$, are perfect squares. A method is known for generating numerical examples of such sets when  $m=2$ or 3 and $n$ is arbitrary. When both $m$ and $n$ exceed 2, only one two-parameter solution with $(m, n)=(4, 4)$ has been published. In this paper we obtain several multi-parameter solutions of the problem in three cases when $(m, n)$ is $(3, 3)$ or $(5, 3)$ or $(4, 4)$, and we indicate  how more such solutions may be obtained.

Keywords: perfect squares; semi-magic matrices with squared entries. 

Mathematics Subject Classification 2020: 11D09
 
\section{Introduction}
This paper is concerned with the  problem of finding two  sets of  integers, $A= \{a_1, a_2, \ldots$, $a_m\}$ and $B=\{b_1, b_2, \ldots, b_n\}$,  such that all elements of the sumset $A+B$, that is, the  $mn$ sums $a_i+b_j, i=1, \ldots, m, j=1, \ldots, n$, are perfect squares. The problem  has been mentioned in Guy's book, `Unsolved Problems in Number Theory' \cite[Problem D14, pp. 266-267]{Gu} where a method is described for  generating numerical solutions of the problem  when $m=2$ or 3, and $n$ is arbitrary.

When $(m, n)=(4, 4)$,  Guy\cite[p. 267]{Gu} has presented a numerical solution, as well as a parametric solution found by Jean Lagrange 
by simply stating the related   $4 \times 4$ matrices of squares whose $(i, j)^{\rm th}$ entry is $a_i+b_j$. We thus get the numerical quadruples $A=\{324,  54756,  119716,  264196\}$ and $B=\{0, 79200,  227205,  1258560\}$, and also a solution in terms of arbitrary parameters $u$ and $v$ given by
\begin{equation}
\begin{aligned}
a_1   & = 9(18u^2 + 50uv + 37v^2)^2, \\
a_2   & = 36(u^2 + 13uv + 16v^2)^2, \\
a_3   & = 36(9u^2 + 53uv + 64v^2)^2, \\
a_4   & = 4(3u^2 - 25uv - 48v^2)^2, \\
b_1   & =0, \\
b_2   & = 20(u - 4v)(11u + 19v)(u + 4v)(u + v), \\
b_3   & = 36(3u + 11v)(7u + 12v)(3u + 4v)(u + v),\\
 b_4   & = 756(u - v)(3u + 2v)(u + 3v)(u + 2v),
\end{aligned}
\label{Lagrangesol}
\end{equation}
such that all the 16 sums $a_i+b_j$ are squares. Apart from the above  solution, no other parametric solutions seem to have been published for any other values of $(m, n)$. 

In this paper we obtain several multi-parameter solutions of the problem when $(m, n)$ is $(3, 3)$ or $(5, 3)$ or $(4, 4)$ and we show how more such parametric solutions may be obtained. We obtain all our parametric solutions by using semi-magic matrices that satisfy certain conditions. In Section  \ref{prelemma} we prove a couple of preliminary lemmas, and in the next three sections we apply the lemmas to obtain parametric solutions for the three cases mentioned above.  We conclude the paper with a couple of open problems.

	\section {Preliminary lemmas and remarks}\label{prelemma}
	We first give a  preliminary lemma about two sets of rational numbers, namely $A=\{a_i, i=1, \ldots, m\}$  and $B=\{b_j, j=1, \ldots, n\}$, such that  all the $mn$  sums $a_i+ b_j, i=1, \ldots, m$, and $j=1, \ldots, n$, are   squares of rational numbers. It immediately follows from this lemma that  to obtain  two sets of integers  $\{a_i, i=1, \ldots, m\}$  and $\{b_j, j=1, \ldots, n\}$ such that all the $mn$ sums $a_i+b_j$ are squares, it suffices to obtain two sets of rational numbers with the same property.   
	
\begin{lemma}\label{canonicallemma} 	If there exist two sets of rational numbers, namely $A=\{a_i, i=1, \ldots, m\}$  and $B=\{b_j, j=1, \ldots, n\}$, such that all the $mn$ sums $a_i+ b_j, i=1, \ldots, m$, and $j=1, \ldots, n$, are  squares, then if $t$ is an arbitrary integer and $k$ is either an arbitrary integer or a rational number so chosen that $k^2a_i, i=1, \ldots, m$,  and $k^2b_j, j=1, \ldots, n$, are all integers, and we define $a^{\prime}_i={k^2a_i-t, i=1, \ldots, m}$, $b^{\prime}_j={k^2b_j+t, i=1, \ldots, n}$, then the two sets $A^{\prime}=\{a^{\prime}_i, i=1, \ldots, n\}$  and $B^{\prime}=\{b^{\prime}_j, j=1, \ldots, n\}$  consist of integers such that all the $mn$ sums $a^{\prime}_i+ b^{\prime}_j, i=1, \ldots, m$, and $j=1, \ldots, n$, are  perfect squares.
\end{lemma}
\begin{proof} The numbers $a^{\prime}_i, i=1, \ldots, m$ 	and $b^{\prime}_j, j=1, \ldots, n$, are clearly integers. Further, the $mn$ sums $a^{\prime}_i+ b^{\prime}_j$ may be written as $k^2(a_i+ b_j), i=1, \ldots, m$, and $j=1, \ldots, n$  and are hence perfect squares. This proves the lemma.
\end{proof}

In view of the above lemma,  two such sets of integers, $A$ and $B$, can be presented in a standard form in which all the integers of the two sets are nonnegative, they do not have a common squared factor, the integers of both the sets are in ascending order of magnitude and  the smallest integer of the set $B$ is 0. In this standard form,   all the integers 	 of the set $A$ are necessarily  perfect squares. The numerical examples of such sets  obtained in this paper will be presented in this standard  form with $b_1=0$. The sets $A$ and $B$ that we obtain in terms of certain arbitrary parameters will be presented generally, but not always, with $b_1=0$ and such that the elements  of the two sets  do not have a common squared factor.

\hspace{0.25in}	It would be recalled that a square matrix is said to be a semi-magic matrix if the sum of the entries in each row and each column is the same. This common sum is referred to as the semi-magic sum. The following lemma relates the existence of two sets of integers  $A=\{a_i, i=1, \ldots, m\}$ and $B=\{b_i, j=1, \ldots, n\}$, where $(m, n)=(3, 3)$ or (5, 3) or (4, 4), such that all the $mn$  sums $a_i +b_j$ are perfect squares  to certain  $ 3 \times 3 $ semi-magic matrices  which  satisfy certain conditions.

\begin{lemma}\label{onelem}
There exist two triples of integers $A=\{a_1,\,a_2,\,a_3\}$ and $B=\{b_1,\,b_2,\,b_3\}$  such that all the nine sums $a_i +b_j, i=1, 2, 3, j=1, 2, 3$, are squares if and only if there exists a $3 \times 3$ semi-magic integer matrix $E$ with squared entries. When such a matrix $E$ exists and is defined by 
\begin{align}
E=\begin{bmatrix}e_{11}^2&e_{12}^2&e_{13}^2\\e_{21}^2 &e_{22}^2 &e_{23}^2 \\ e_{31}^2 &e_{32}^2 &e_{33}^2 \end{bmatrix},
\label{matrixE}
\end{align}
then 
the two triples of integers $A=\{a_1,\,a_2,\,a_3\}$ and $B=\{b_1,\,b_2,\,b_3\}$ defined by
\begin{equation}
\begin{aligned}
a_1 & =  e_{13}^2, &   a_2 & = e_{31}^2, &   a_3 & =  e_{22}^2, \\
 b_1 & =  0, &   b_2 & = e_{21}^2 -e_{13}^2 , &   b_3 & =  e_{11}^2 - e_{22}^2,
\end{aligned}
\label{defAB}
\end{equation}
are such that all the nine sums $a_i +b_j, i=1, 2, 3, j=1, 2, 3$, are squares.

Further, the two triples $A$ and $B$ defined by \eqref{defAB} may be extended to a quintuple $A_1=\{a_1,\,a_2,\,a_3, a_4, a_5\}$ and a triple $B_1=\{b_1,\,b_2,\,b_3\}$ such that all the 15 sums $a_i +b_j, i=1, \ldots, 5, j=1, 2, 3$, are squares if and only if, in addition to the matrix $E$ defined by \eqref{matrixE}, there exists a $ 3 \times 3$ semi-magic integer matrix $F$ with squared entries, namely 
\[
F=\begin{bmatrix}f_{11}^2&f_{12}^2&f_{13}^2\\f_{21}^2 &f_{22}^2 &f_{23}^2 \\ f_{31}^2 &f_{32}^2 &f_{33}^2 \end{bmatrix}, \label{matrixF}
\] 
such that the entries of the two matrices $E$ and $F$ satisfy the following conditions:
\begin{equation}
e_{11}^2=f_{11}^2,  \quad e_{22}^2=f_{22}^2, \quad e_{33}^2=f_{33}^2.
\label{relef}
\end{equation}
When two such matrices $E$ and $F$ exist, and $a_i, b_i, i=1, 2, 3$, are  defined, as before, by \eqref{defAB} and $a_4, a_5$ are defined by
\begin{equation}
a_4 = f_{13}^2, \quad a_5 = f_{31}^2, \label{vala45}
\end{equation}
then  all the 15 sums $a_i+b_j, i=1, \ldots, 5, b_j =1, 2, 3$, are squares.

Finally,  the two triples of integers $A$ and $B$ defined by \eqref{defAB} may be extended to two quadruples of integers $A_2=\{a_1, a_2, a_3, a_4\}$ and $B_2=\{b_1, b_2, b_3, b_4\}$ such that all the 16 sums $a_i +b_j, i=1, \ldots, 4, j=1, \ldots, 4$ are  squares  if and only if, in addition to the matrix $E$, there exist  two   $3 \times 3$ semi-magic integer matrices $G$ and $H$ with squared entries, namely
\[
G=\begin{bmatrix}g_{11}^2&g_{12}^2&g_{13}^2\\g_{21}^2 &g_{22}^2 &g_{23}^2 \\ g_{31}^2 &g_{32}^2 &g_{33}^2 \end{bmatrix},\;
H=\begin{bmatrix}h_{11}^2&h_{12}^2&h_{13}^2\\h_{21}^2 &h_{22}^2 &h_{23}^2 \\ h_{31}^2 &h_{32}^2 &h_{33}^2 \end{bmatrix},\; 
\label{matrixGH}
\]
such that  the entries of the three matrices $E, G$ and $H$ satisfy the following conditions:
\begin{align}
e_{12}^2 &  =g_{12}^2 , & e_{13}^2 &  =g_{13}^2 ,  &e_{21}^2 &  =g_{21}^2 ,  &e_{31}^2 &  =g_{31}^2 , \label{releg} \\
e_{11}^2 &  =h_{11}^2 , & e_{22}^2 &  =h_{22}^2 , &&&& \label{releh} \\
g_{11}^2 &  =h_{12}^2 , & g_{22}^2 &  =h_{31}^2 . &&&& \label{relgh}
\end{align}
When such matrices $E, G$ and $H$ exist, and $a_i, b_i, i=1, 2, 3$, are  defined, as before, by \eqref{defAB} and $a_4, b_4$ are defined by
\begin{equation}
 a_4  =  g_{22}^2, \quad b_4  =  g_{11}^2 - g_{22}^2,
\label{vala4b4}
\end{equation}
the two quadruples of integers $A_2=\{a_1, a_2, a_3, a_4\}$ and $B_2=\{b_1, b_2, b_3, b_4\}$
are such that all the 16 sums $a_i +b_j, i=1, \ldots, 4, j=1, \ldots, 4$,       are  squares.
\end{lemma}

\begin{proof}
If there exist two triples of integers $\{a_i, a_2, a_3\}$ and $\{b_i,b_2, b_3\}$ such that all the nine sums $a_i +b_j$ are perfect squares, then it is readily seen that the matrix 
\begin{align}
E_1 & =\begin{bmatrix} a_3 + b_3 &  a_2 + b_2 &  a_1 + b_1 \\  a_1 + b_2 &  a_3 + b_1 &  a_2 + b_3 \\  a_2 + b_1 &  a_1 + b_3 &  a_3 + b_2 \end{bmatrix}
\label{matrixE1}
\end{align}
is a semi-magic matrix with common sum $a_1+a_2+a_3+b_1+b_2+b_3$, and all the entries of the matrix $E$ are perfect squares. 

Conversely, if the matrix $E$  defined by \eqref{matrixE} is semi-magic, we will show that  the integers $a_i, b_i$ defined by \eqref{defAB} are such that all the nine sums $a_i+b_j, i=1, 2, 3, j=1, 2, 3$, are squares. 

Since $b_1=0$, the three sums $a_i+b_1, i=1, \ldots, 3$, are clearly  squares.  Further, since $E$ is a semi-magic matrix,  we have,  
\begin{multline*}
\begin{aligned}
a_1+b_2 &=  e_{ 21}^2, \quad & a_1+b_3&=e_{ 13}^2+ e_{ 11}^2  - e_{ 22}^2 =  e_{ 32}^2,\\
a_2+b_2 &=e_{31}^2+e_{21}^2-e_{13}^2= e_{ 12}^2 , \quad & a_2+b_3 &= e_{31}^2 + e_{ 11}^2 - e_{ 22}^2= e_{ 23}^2, \\
a_3+b_2 & =e_{22}^2+ e_{21}^2 -e_{13}^2   = e_{33}^2, \quad & a_3+ b_3 &= e_{11}^2.
\end{aligned}
\end{multline*}
Thus, all the nine sums $a_i+b_j, i=1, 2, 3, j=1, 2, 3$, are squares. This proves the first part of the lemma concerning the two triples $A$ and $B$. 

Next, we note that if there exists a quintuple $A_1=\{a_1,\ldots, a_5\}$ and a triple $B_1=\{b_1,\,b_2,\,b_3\}$ such that all the 15 sums $a_i+b_j, i=1, \ldots, 5,  j=1, 2, 3$, are  squares, then all the entries of the two matrices $E_1$,  defined by \eqref{matrixE1}, and the matrix $F_1$  defined by 
\begin{align}
F_1  =\begin{bmatrix} a_3 + b_3 &  a_5 + b_2 &  a_4 + b_1 \\  a_4 + b_2 &  a_3 + b_1 &  a_5 + b_3 \\  a_5 + b_1 &  a_4 + b_3 &  a_3 + b_2 \end{bmatrix},
\end{align}
are squares. Further, both $E_1$ and $F_1$ are semi-magic matrices, their common sums being $a_1+a_2+a_3+b_1+b_2+b_3$ and $a_3+a_4+a_5+b_1+b_2+b_3$, respectively. Finally, it is readily seen that the entries of the matrices $E_1$ and $F_1$ satisfy the conditions \eqref{relef}. 

Conversely, let $E=[e_{ij}^2]$ and $F=[f_{ij}^2]$ be two  semi-magic matrices such that  the conditions \eqref{relef} are satisfied.
We will now show that the integers $\{a_1, \ldots, a_5\}$ and $\{b_1, b_2, b_3\}$ defined by \eqref{defAB} and \eqref{vala45} are such that all the 15 sums $a_i+b_j, i=1, \ldots, 5,  j=1, 2, 3$, are  squares. 

The nine sums $a_i+ b_j,i=1, 2, 3, j=1, 2, 3$, are squares as proved earlier. To prove that the remaining six sums are squares, we first note that since $E$ is a semi-magic matrix, the sum of the second row and the third column of $E$ is the same, hence $e_{21}^2+e_{22}^2=e_{13}^2+e_{33}^2$, and thus $b_2=e_{21}^2-e_{13}^2=e_{33}^2-e_{22}^2$. In the following proof, it will be convenient to use the latter value of $b_2$. We now note that 
\begin{equation}
\begin{aligned}
a_4+b_1 & = f_{13}^2, & & &&\\
a_4+b_2 & = f_{13}^2+ e_{33}^2-e_{22}^2 & =&f_{13}^2+ f_{33}^2-f_{22}^2 & = & f_{21}^2, \\ 
a_4+b_3 & = f_{13}^2+e_{11}^2 - e_{22}^2 & = & f_{13}^2+f_{11}^2 - f_{22}^2 & = &f_{32}^2,\\
a_5+b_1 & = f_{31}^2, &&&& \\
a_5+b_2 & =f_{31}^2+ e_{33}^2-e_{22}^2 & = &f_{31}^2+ f_{33}^2-f_{22}^2 &= & f_{12}^2, \\
a_5+b_3 & = f_{31}^2+e_{11}^2 - e_{22}^2 & =&  f_{31}^2+f_{11}^2 - f_{22}^2 &= &f_{23}^2. 
	\end{aligned}
\end{equation}
Thus, all the 15 sums $a_i +b_j, $i=1, \ldots, 5$, j=1, 2, 3$, are squares. This proves the second part of the lemma concerning a quintuple  and a triple.

Finally we note that  if there exist two quadruples of integers $\{a_i, i=1, \ldots, 4\}$ and $\{b_i, j=1, \ldots, 4\}$  such that all the 16 sums $a_i +b_j, i=1, \ldots, 4, j=1, \ldots, 4$,  are perfect squares, then it is readily seen that the  three matrices $E_1$, defined by \eqref{matrixE1},  and the matrices $G_1$ and $H_1$ defined by 
\begin{align}
G_1 & =\begin{bmatrix} a_4 + b_4 &  a_2 + b_2 &  a_1 + b_1 \\  a_1 + b_2 &  a_4 + b_1 &  a_2 + b_4 \\  a_2 + b_1 &  a_1 + b_4 &  a_4 + b_2 \end{bmatrix},\\[0.1in]
H_1 & =\begin{bmatrix} a_3 + b_3 &  a_4 + b_4 &  a_1 + b_1 \\  a_1 + b_4 &  a_3 + b_1 &  a_4 + b_3 \\  a_4 + b_1 & a_1 + b_3 &  a_3 + b_4 \end{bmatrix},
\end{align} 
are such that all the entries of these three matrices are perfect squares, all of them are semi-magic matrices (their common sums being $a_1+a_2+a_3+b_1+b_2+b_3$, $a_1+a_2+a_4+b_1+b_2+b_4$ and $a_1+a_3+a_4+b_1+b_3+b_4$, respectively) and  the relations \eqref{releg},  \eqref{releh} and  \eqref{relgh} are all satisfied.

Conversely, if there exist three semi-magic integer matrices $E=[e_{ij}^2], G =[g_{ij}^2], H=[h_{ij}^2]$, with squared entries satisfying the conditions \eqref{releg}, \eqref{releh} and  \eqref{relgh}, we will  show that the  integers  $a_i$ and $b_j$  defined by \eqref{defAB}  and \eqref{vala4b4}  are such that  all the 16  sums $a_i +b_j, i=1, \ldots, 4, j=1, \ldots, 4$, are perfect squares. 

We note that the nine sums $a_i+ b_j,i=1, 2, 3, j=1, 2, 3$, are squares as proved earlier. Next, using the relations \eqref{defAB}, \eqref{vala4b4} and the fact that $G$ is a semi-magic matrix, we have,
\begin{equation*}
\begin{aligned}
a_1+b_4 & =e_{13}^2 + g_{11}^2 - g_{22}^2 =g_{13}^2 + g_{11}^2 - g_{22}^2=g_{32}^2, \\
a_2+b_4 & = e_{31}^2 + g_{11}^2 - g_{22}^2 =g_{31}^2+g_{11}^2-g_{22}^2 = g_{23}^2, \\
a_4+b_1 & = g_{22}^2, \\
a_4+b_2 & =g_{22}^2 + e_{21}^2-e_{13}^2 = g_{22}^2 +g_{21}^2 -g_{13}^2 =g_{33}^2, \\
a_4+b_4 & = g_{11}^2. \\
\end{aligned}
\end{equation*}

Finally, using the relations \eqref{releh} and \eqref{relgh}, and the fact that $H$ is a semi-magic matrix, we have,
\begin{equation*}
\begin{aligned}
a_3+b_4 & = e_{22}^2 + g_{11}^2 - g_{22}^2 = h_{22}^2 + h_{12}^2-h_{31}^2=h_{33}^2,\\
a_4+b_3 & =g_{22}^2 + e_{11}^2 - e_{22}^2  =h_{31}^2 +  h_{11}^2 - h_{22}^2 = h_{23}^2.
\end{aligned}
\end{equation*}

Thus all the 16  sums $a_i +b_j, i=1, \ldots, 4, j=1, \ldots, 4$, are  squares. This proves the final  part of the lemma concerning the two quadruples   and the proof is complete.
\end{proof} 
 
In the next three sections, we will obtain  two sets of integers $\{a_1, a_2, \ldots, a_m\}$ and $\{b_1, b_2, \ldots, a_n\}$  in the three cases when $(m, n)$ is either (3, 3) or (5, 3) or (4, 4), respectively, such that, in each case, all the  $mn$ sums $a_i +b_j, i=1, \ldots, m$ and $j=1, \ldots, n$, are squares by constructing semi-magic integer matrices satisfying the conditions of Lemma \ref{onelem}. For this purpose,  we will repeatedly use the following semi-magic matrix of squares given by Euler (as quoted by Dickson \cite[p. 530]{Di}):
\small
\begin{equation}
M=\begin{bmatrix} (p^2+q^2-r^2-s^2)^2 & (2qr+2ps)^2 & (2qs-2pr)^2 \\(2qr-2ps)^2 & (p^2-q^2+r^2-s^2)^2 & (2pq+2rs)^2 \\(2qs+2pr)^2 & (2rs-2pq)^2 & (p^2-q^2-r^2+s^2)^2 \end{bmatrix}. \label{Euler}
\end{equation}
\normalsize:

\section{Two triples $\{a_1, a_2, a_3\}$ and $\{b_1, b_2,b_3\}$  such that all the nine sums $a_i+b_j$ are squares}\label{sets33}
The following theorem gives three pairs of triples $\{a_1, a_2, a_3\}$ and $\{b_1, b_2, b_3\}$, in terms of four parameters, such that all the nine sums $a_i+b_j$ are perfect squares.

\begin{theorem}If $\{a_1, a_2, a_3\}$ and $\{b_1, b_2, b_3\}$ are two triples of integers, defined in terms of arbitrary integer parameters $p, q, r$ and $s$, either by
\begin{equation}
\begin{aligned}
a_1 & =  0, \\
a_2 & =  16pqrs, \\
a_3 & =  (p + q + r + s)(p + q - r - s)(p - q + r - s)(p - q - r + s), \\
b_1 & =  4(pr - qs)^2,\\
 b_2 & =  4(ps - qr)^2,\\
 b_3 & =  4(pq - rs)^2,
\end{aligned}
\label{triplepair1}
\end{equation}
or by
\begin{equation}
\begin{aligned}
a_ 1 & =  -(p + q + r + s)(p + q - r - s)(p - q + r - s)(p - q - r + s), \\
a_ 2 & =  -(p - q + r + s)(p - q - r - s)(p + q + r - s)(p + q - r + s), \\
a_ 3 & =  0, \\
b_ 1 & =  (p^2 - q^2 + r^2 - s^2)^2,\\
 b_ 2 & =  (p^2 - q^2 - r^2 + s^2)^2,\\
 b_ 3 & =  (p^2 + q^2 - r^2 - s^2)^2,
\end{aligned}
\label{triplepair2}
\end{equation}
or by
\begin{equation}
\begin{aligned}
a_1 & = 4(ps - qr)^2, & a_2 & = 4(ps + qr)^2, \\
a_3 & = (p^2 - q^2 - r^2 + s^2)^2, & b_1 & = 4(p^2 - q^2)(r^2 - s^2),\\
 b_2 & = 0, & b_3 & = 4(p^2 - r^2)(q^2 - s^2),
\end{aligned}
\label{triplepair3}
\end{equation}
then for each of these three pairs of triples $\{a_1, a_2, a_3\}$ and $\{b_1, b_2, b_3\}$, all the nine sums $a_i+b_j, i=1, 2, 3, j=1, 2, 3$, are perfect squares. 
\end{theorem}
\begin{proof}
We first apply Lemma \ref{onelem} taking $E$ to be the matrix $M$ 
defined by \eqref{Euler} and, from the relations \eqref{defAB} we obtain two triples $\{a_1, a_2, a_3\}$ and $\{b_1, b_2,b_3\}$. Next, in accordance with Lemma \ref{canonicallemma}, we subtract $t$ from each of the integers $a_i$ and we add $t$ to each of the integers $b_i$ to get the following two triples, in  terms of arbitrary integer parameters $p, q, r, s$ and $t$ such that all the nine sums $a_i+b_j, i=1, 2, 3, j=1, 2, 3$, are squares:
\begin{equation}
\begin{aligned}
a_1 &= 4(pr - qs)^2-t,\\ 
a_2 &= 4(pr + qs)^2-t, \\
a_3 &= (p^2 - q^2 + r^2 - s^2)^2 - t,\\
 b_1 &= t, \\
b_2 &= t - 4(p^2 - q^2)(r^2 - s^2),\\
 b_3 &= t + 4(p^2 - s^2)(q^2 - r^2).
\end{aligned}
\end{equation}

Assigning to $t$ the three values $4(pr - qs)^2, (p^2 - q^2 + r^2 - s^2)^2$ and $4(p^2 - q^2)(r^2 - s^2)$, in succession, we get the three pairs of triples defined by \eqref{triplepair1}, \eqref{triplepair2} and \eqref{triplepair3}, respectively, such that in each case, all the nine sums $a_i+b_j, i=1, 2, 3, j=1, 2, 3$, are squares. This proves the lemma.
\end{proof}

As a numerical example, taking $(p, q, r, s)= (5, 3, 2, 1)$ in the triples defined by \eqref{triplepair3}, we get the two triples $A= \{ 4, 169,  484\}$ and $\{0,  192, 672\}$ such that all the nine elements of the sumset $A+B$ are perfect squares. 

\section{A quintuple and a triple}

We will now apply Lemma \ref{onelem} to obtain a quintuple $a_1, \ldots, a_5$ and a triple $b_1, b_2, b_3$ such that all the 15 sums $a_i+b+j$ are squares. The result thus obtained is stated in the following theorem.

\begin{theorem}\label{thm35}
If $a_1, \ldots, a_5$ and $b_1, b_2, b_3$ are rational numbers, defined in terms of arbitrary rational parameters $f, g, m, u_1, u_2, v_1$,and $v_2$ by
\begin{equation}
\begin{aligned}
a_1 & = 4(mu_1u_2(u_1v_1 - u_2v_2)f^2 - (u_1v_2 - u_2v_1)\\
&  \quad \quad \times (m^2u_1u_2v_1v_2 - 1)fg - mv_1v_2(u_1v_1 - u_2v_2)g^2)^2, \\
a_2 & = 4(mu_1u_2(u_1v_1 + u_2v_2)f^2 + (u_1v_2 + u_2v_1)\\
&  \quad \quad \times (m^2u_1u_2v_1v_2 + 1)fg + mv_1v_2(u_1v_1 + u_2v_2)g^2)^2, \\
a_3 & = ((u_1^2u_2^2(v_1^2 - v_2^2)m^2 + u_1^2 - u_2^2)f^2\\
&  \quad \quad - (v_1^2v_2^2(u_1^2 - u_2^2)m^2 + v_1^2 - v_2^2)g^2)^2, \\
 a_4 & = 4(mu_1u_2(u_1v_1 - u_2v_2)f^2 + (u_1v_2 - u_2v_1)\\
&  \quad \quad \times (m^2u_1u_2v_1v_2 - 1)fg - mv_1v_2(u_1v_1 - u_2v_2)g^2)^2, \\
a_5 & = 4(mu_1u_2(u_1v_1 + u_2v_2)f^2 - (u_1v_2 + u_2v_1)\\
&  \quad \quad \times (m^2u_1u_2v_1v_2 + 1)fg + mv_1v_2(u_1v_1 + u_2v_2)g^2)^2,\\
b_1 & =0,\\
b_2 & = -4(u_1^2-u_2^2)(v_1^2-v_2^2)(m^2u_1^2u_2^2f^2 - g^2)(f^2-m^2v_1^2v_2^2g^2),\\
b_3 & = 4(m^2u_1^2v_1^2 - 1)(m^2u_2^2v_2^2 - 1)(f^2u_1^2 - g^2v_1^2)(f^2u_2^2 - g^2v_2^2)
	\end{aligned}
	\label{thm35valab}
\end{equation}
all the 15 sums $a_i +b_j, $i=1, \ldots, 5$, j=1, 2, 3$, are squares of rational numbers. 
\end{theorem}
\begin{proof}
We will obtain the rational numbers $a_1, \ldots, a_5$ and $b_1, b_2, b_3$ by first finding two semi-magic matrices $E$ and $F$ of squares that satisfy the conditions \eqref{relef} and then using the relations \eqref{defAB} and \eqref{vala45} given in   Lemma \ref{onelem}. We begin with two matrices $E$ and $F$ obtained by  replacing  the parameters $p, q, r, s$ in Euler's  matrix $M$ by $p_1, q_1, r_1, s_1$ and $p_2, q_2, r_2, s_2$, respectively, that is, with the matrices,
\begin{equation}
E=\begin{bmatrix} (p_1^2+q_1^2-r_1^2-s_1^2)^2 & (2q_1r_1+2p_1s_1)^2 & (2q_1s_1-2p_1r_1)^2 
\\(2q_1r_1-2p_1s_1)^2 & (p_1^2-q_1^2+r_1^2-s_1^2)^2 & (2p_1q_1+2r_1s_1)^2 \\
(2q_1s_1+2p_1r_1)^2 & (2r_1s_1-2p_1q_1)^2 & (p_1^2-q_1^2-r_1^2+s_1^2)^2 \end{bmatrix}, \label{lem35E}
\end{equation}
and
\begin{equation}
F=\begin{bmatrix} (p_2^2+q_2^2-r_2^2-s_2^2)^2 & (2q_2r_2+2p_2s_2)^2 & (2q_2s_2-2p_2r_2)^2 \\
(2q_2r_2-2p_2s_2)^2 & (p_2^2-q_2^2+r_2^2-s_2^2)^2 & (2p_2q_2+2r_2s_2)^2 \\
(2q_2s_2+2p_2r_2)^2 & (2r_2s_2-2p_2q_2)^2 & (p_2^2-q_2^2-r_2^2+s_2^2)^2 \end{bmatrix}. \label{lem35F}
\end{equation}
The matrices $E$ and $F$ will satisfy the conditions \eqref{relef} if the following conditions are satisfied:
\begin{align}
p_1^2+q_1^2-r_1^2-s_1^2 & = p_2^2+q_2^2-r_2^2-s_2^2, \label{thm35eq1}\\
p_1^2-q_1^2+r_1^2-s_1^2 & = p_2^2-q_2^2+r_2^2-s_2^2, \label{thm35eq2}\\
p_1^2-q_1^2-r_1^2+s_1^2 & = p_2^2-q_2^2-r_2^2+s_2^2. \label{thm35eq3}
\end{align}
Both Eqs. \eqref{thm35eq2} and \eqref{thm35eq3} will be satisfied if $p_1^2-q_1^2 = p_2^2-q_2^2$ and $r_1^2-s_1^2=r_2^2-s_2^2$. Accordingly, a solution of the simultaneous equations \eqref{thm35eq2} and \eqref{thm35eq3} may readily be  written as follows:
\begin{equation}
\begin{aligned}
p_ 1  & =  eu_1+fu_2, \quad  & p_2  & = eu_1-fu_2, \\
 q_1  & =  eu_2+fu_1, \quad  & q_2  & =  eu_2-fu_1, \\
 r_1  & =  gv_1+hv_2, \quad  & r_2  & =  gv_1-hv_2, \\
s_1  & =  gv_2+hv_1, \quad  & s_2  & =  gv_2-hv_1,
\end{aligned}
\label{thm35valpqrs}
\end{equation}
where $e, f, g, h, u_1, u_2, v_1$ and $v_2$ are arbitrary parameters.

Substituting the values of $p_i, q_i$, $r_i, s_i, i=1, 2$, in Eq. \eqref{thm35eq1},  and  transposing all terms to the left-hand side, we get the condition,
\begin{equation}
(efu_1u_2 - ghv_1v_2)((e^2 + f^2)(u_1^2 + u_2^2) -(g^2+h^2) (v_1^2 + v_2^2))=0.
\label{condeq1}
\end{equation}
For simplicity, we solve Eq. \eqref{condeq1} by equating to 0 the first factor on the left-hand side of Eq. \eqref{condeq1}, and accordingly, we take
\begin{equation}
e = gmv_1v_2, \quad h = fmu_1u_2, \label{thm35valeh}
\end{equation}
where $m$ is an arbitrary rational parameter. 

Thus, when the values of $p_i, q_i, r_i, s_i, i=1, 2$, are given by \eqref{thm35valpqrs} and the values of $e$ and $h$ are given by \eqref{thm35valeh}, then the semi-magic matrices $E$ and $F$ defined by \eqref{lem35E} and \eqref{lem35F}, respectively, satisfy all the conditions \eqref{relef}. Now, on applying Lemma \ref{onelem}, we get the rational numbers $a_1, \ldots, a_5$ and $b_1, b_2, b_3$ defined by \eqref{thm35valab} such that all the 15 sums  $a_i +b_j, $i=1, \ldots, 5$, j=1, 2, 3$, are squares of rational numbers. This proves the theorem. 
\end{proof}

As a numerical example, on taking $(f, g, m, u_1, u_2, v_1, v_2)=(4, 1, 3, 2, 1, 3, 1)$ in the solution \eqref{thm35valab}, we get the quintuple 
\[
A=\{198916, 532900, 1674436, 13468900, 19404025\}
\]
 and the triple 
\[B=\{ 0, 3588000,  8527200\}
\]
 such that all the 15 elements of the sumset $A+B$ are perfect squares. 

We note that more examples of a quintuple and a triple may be obtained by equating to 0 the second factor on the left-hand side of Eq. \eqref{condeq1} and proceeding as before. As these solutions are more cumbersome to write, we omit them. 

\section{Two quadruples of integers $\{a_i, i=1, \ldots, 4\}$ and $\{b_i, j=1, \ldots, 4\}$ such that all sixteen sums $a_i +b_j$ are perfect squares.}
We will now obtain three examples of two quadruples $A=\{a_i, i=1, \ldots, 4\}$ and $B=\{b_i, j=1, \ldots, 4\}$, in parametric terms, such that all the 16  sums $a_i +b_j, i=1, \ldots, 4, j=1, \ldots, 4$, are perfect squares. We also show how more such examples of quadruples may be constructed.

We can, and of course we will, apply Lemma \ref{onelem} to find such quadruples. We, however, first note that if, instead of constructing the three semi-magic matrices $E, G$ and $H$ as stipulated in Lemma \ref{onelem}, we just construct two semi-magic  matrices $E$ and $G$ such that the conditions \eqref{releg} are satisfied, and we use the relations \eqref{defAB} and \eqref{vala4b4} to construct two quadruples $\{a_i, i=1, \ldots, 4\}$ and $\{b_i, j=1, \ldots, 4\}$, then except for the two sums $a_3+b_4$ and $a_4+b_3$, all the remaining 14 sums $a_i+b_j$ are perfect squares. This follows from a perusal of the proof of Lemma \ref{onelem} where the matrix $H$ and the conditions \eqref{releh} and \eqref{relgh} are needed only  to ensure that the two sums $a_3+b_4$ and $a_4+b_3$ are squares --- the matrices $E$ and $G$ and the relations \eqref{releg} being sufficient to prove that  14 of the sums $a_i+b_j$ are squares. 

In Section \ref{quadruplesfirstpair} we obtain, in parametric terms, two semi-magic matrices $E$ and $G$ with squared entries such that the conditions \eqref{releg} are satisfied. We then use the relations  \eqref{defAB} and \eqref{vala4b4} to obtain two quadruples   $\{a_i, i=1, \ldots, 4\}$ and $\{b_i, j=1, \ldots, 4\}$ such that 14 of the sums $a_i+b_j$ are squares. Finally, we choose the parameters such that the two remaining sums $a_3+b_4$ and $a_4+b_3$ also become squares.

In Section \ref{twomorepairs} we use the matrices $E$ and $G$ of Section \ref{quadruplesfirstpair} and construct the matrix $H$ such that all the conditions \eqref{releg}, \eqref{releh} and \eqref{relgh} are satisfied, and we apply Lemma \ref{onelem} to obtain two pairs of quadruples $A$ and $B$ such that all elements of the sumset $A+B$ are squares.

\subsection{A pair of quadruples} \label{quadruplesfirstpair} 
\begin{theorem}\label{thm44first}
If $a_1, \ldots, a_4$ and $b_1, \ldots, b_4$ are integers, defined in terms of arbitrary integer parameters $u_i, v_i, i=1, 2$,  by
\begin{equation}
\begin{aligned}
a_1  &  =  (v_1^4 - 6v_1^2v_2^2 + v_2^4)^2(v_1^2-v_2^2)^2(u_1^2 + u_2^2)^2, \\
a_2  &  =  (v_1^4 - 6v_1^2v_2^2 + v_2^4)^2(v_1^2+v_2^2)^2(u_1^2 - u_2^2)^2,  \\
a_3  &  =  (v_1^2-v_2^2)^2((v_1^2+v_2^2)^2u_1^2-16u_1u_2v_1^2v_2^2-(v_1^2+v_2^2)^2u_2^2)^2,  \\
a_4  &  =  (v_1^2-v_2^2)^2((v_1^2+v_2^2)^2u_1^2+16u_1u_2v_1^2v_2^2-(v_1^2+v_2^2)^2u_2^2)^2,  \\
b_1  &  =  0,  \\
b_2  &  =  (v_1^4 - 6v_1^2v_2^2 + v_2^4)^2((v_2-v_1)u_1+(v_1+v_2)u_2)\\
 &  \quad \quad \times((v_1+v_2)u_1+(v_1-v_2)u_2)((v_1-v_2)u_1+(v_1+v_2)u_2)\\
 &  \quad \quad \times ((v_1+v_2)u_1-(v_1-v_2)u_2),\\
b_3   &  =   -(v_1^2 - v_2^2)^2((v_1^2 + v_2^2)u_1+ 4u_1v_1v_2 -(v_1^2+ v_2^2)u_2)\\
& \quad \quad \times ((v_1^2 + v_2^2)u_1- 4u_1v_1v_2 -(v_1^2+ v_2^2)u_2)\\
& \quad \quad \times ((v_1^2 + v_2^2)u_1  + 4u_2v_1v_2+ (v_1^2+ v_2^2)u_2)\\
& \quad \quad \times ((v_1^2 + v_2^2)u_1 - 4u_2v_1v_2+ (v_1^2+ v_2^2)u_2),\\
b_4   &  =   -(v_1^2 - v_2^2)^2((v_1^2 + v_2^2)u_1+ 4u_1v_1v_2 +(v_1^2+ v_2^2)u_2)\\
& \quad \quad \times ((v_1^2 + v_2^2)u_1- 4u_1v_1v_2 +(v_1^2+ v_2^2)u_2)\\
& \quad \quad \times((v_1^2 + v_2^2)u_1+ 4u_2v_1v_2- (v_1^2+ v_2^2)u_2)\\
& \quad \quad \times((v_1^2 + v_2^2)u_1 - 4u_2v_1v_2- (v_1^2+ v_2^2)u_2),
\end{aligned}
\label{quadpairfirst}
\end{equation}
all the 16 sums $a_i +b_j, $i=1, \ldots, 4$, j=1, \ldots, 4$, are perfect squares.
\end{theorem}
\begin{proof}
As discussed above, we will obtain the two quadruples $\{a_i, i=1, \ldots, 4\}$ and $\{b_i, j=1, \ldots, 4\}$ with all 16  sums $a_i +b_j$ being perfect squares by just constructing, in parametric terms,  two $3 \times 3$  semi-magic matrices $E$ and $G$  with squared entries  such that the conditions \eqref{releg} are satisfied. As before, we will use Euler's matrix $M$ to  construct the desired  semi-magic matrices $E$ and $G$.

 We take the matrix $E$ to be the matrix $M$ itself, that is, $E=M$. To construct the matrix $G$, we   replace $p$ and $q$ in the matrix $M$ by $mp$ and $mq$, respectively, where $m$ is some nonzero rational number, and then divide each  entry of the matrix thus obtained by $m^2$, when we get the semi-magic matrix
\begin{equation}
G=\begin{bmatrix} g_{11}^2  & (2qr +2ps)^2     &  (2qs - 2p r)^2   \\
(2qr - 2p s)^2   & g_{22}^2  & 4(m^2 p q+r s)^2/m^2   \\
(2qs + 2p r)^2   & 4(m^2 p q-r s)^2/m^2   & g_{33}^2
\end{bmatrix}
\label{matG}
\end{equation}
where
\begin{equation}
\begin{aligned}
g_{11} & = (m^2 p^2+m^2 q^2-r^2-s^2)/m, \\
g_{22} & = (m^2 p^2-m^2 q^2+r^2-s^2)/m,\\
g_{33} & = (m^2 p^2-m^2 q^2-r^2+s^2)/m.
\end{aligned}
\label{valgij}
\end{equation}

The entries of the matrices $E$ and $G$ satisfy the relations \eqref{releg}, and using the relations \eqref{defAB} and \eqref{vala4b4}, we get two quadruples $A=\{a_1, a_2, a_3, a_4\}$ and $B = \{b_1, b_2, b_3, b_4\}$ such that except for the two sums $a_3+b_4$ and $a_4+b_3$, all other 14  sums $a_i+b_j$ are perfect squares. The values of $a_i, b_i, i=1, \ldots, 4$, are given in terms of arbitrary parameters $m, p, q, r$ and $s$ by
\begin{equation}
\begin{aligned}
a_1   &  =  4(pr - qs)^2,  \quad &
a_2   &  =  4(pr + qs)^2, \\
a_3   &  =  (p^2 - q^2 + r^2 - s^2)^2, \quad &
 a_4   &  =  (m^2p^2 - m^2q^2 + r^2 - s^2)^2/m^2,\\
 b_1   &  =  0, \quad &
b_2   &  =  -4(p^2-q^2)(r^2 - s^2),\\
 b_3   &  =  4(p^2-s^2)(q^2 - r^2), \quad &
 b_4   &  =  4(m^2p^2 - s^2)(m^2q^2 - r^2)/m^2,
\end{aligned}
\label{quadpair1}
\end{equation}

To  make the two sums $a_3+b_4$ and $a_4+b_3$ also perfect squares, we impose the auxiliary condition $a_3+b_4=a_4+b_3$ which, on using the relations \eqref{quadpair1}, reduces to
\begin{equation}
(m^2 - 1)(mp^2 + mq^2 + r^2 + s^2)(mp^2 + mq^2 - r^2 - s^2)=0. \label{auxeqna3b4}
\end{equation}
A solution of \eqref{auxeqna3b4}, obtained by equating the last factor on the left-hand side of Eq. \eqref{auxeqna3b4}, is given by
\begin{equation}
\begin{aligned}
m & = t^2, \\
 p & = u_1v_1 + u_2v_2, \\
 q & = u_1v_2 - u_2v_1, \\
 r & = t(u_1v_1 - u_2v_2),\\
 s & = t(u_1v_2 + u_2v_1),
\end{aligned}
 \label{valmpq}
\end{equation}
where $t, u_i, v_i, i=1, 2$, are arbitrary parameters.

Now the value of the sum $a_3+b_4$, denoted by $\phi(t, u_1, u_2, v_1, v_2)$, may be worked out  using the relations \eqref{quadpair1} and \eqref{valmpq}, and is given by
\begin{multline}
 \phi(t, u_1, u_2, v_1, v_2)  =(t^2-1)((t^2-1)(v_1^2 + v_2^2)^2u_1^4-16(t^2 + 1)v_1v_2(v_1^2 - v_2^2)u_1^3u_2\\
 \quad \quad \quad \quad \quad \quad \quad \quad +2(t^2-1)(v_1^2 + v_2^2)^2u_1^2u_2^2+16(t^2 + 1)v_1v_2(v_1^2 - v_2^2)u_1u_2^3\\
+(t^2-1)(v_1^2 + v_2^2)^2u_2^4).
\end{multline}
Similarly, the value of the sum $a_4+b_3$ may be worked out and it is given by $\phi(t, u_1, u_2, -v_1, v_2)$. 

To make $\phi(t, u_1, u_2, v_1, v_2)$ a perfect square, we computed its  discriminant $d$ with respect to $u_1$ and equated it to 0. The discriminant of $\phi(t, u_1, u_2, -v_1, v_2)$ with respect to $u_1$ is also $d$, and the value of $d$ is given by
\begin{multline}
d = 262144(t^2 + 1)^2(t^2 - 1)^6u_2^{12}v_1^2v_2^2(v_1^2 - v_2^2)^2(t(v_1^2 - 2v_1v_2 - v_2^2)-v_1^2 - 2v_1v_2 + v_2^2)^2\\
\times (t(v_1^2 - 2v_1v_2 - v_2^2)+v_1^2 + 2v_1v_2 - v_2^2)^2(t(v_1^2 + 2v_1v_2 - v_2^2)-v_1^2 + 2v_1v_2 + v_2^2)^2\\
\times (t(v_1^2 + 2v_1v_2 - v_2^2)+v_1^2 - 2v_1v_2 - v_2^2)^2.
\end{multline}
Equating the factor $(t(v_1^2 - 2v_1v_2 - v_2^2)-v_1^2 - 2v_1v_2 + v_2^2)^2$ to 0, we get
\begin{equation}
t=(v_1^2 + 2v_1v_2 - v_2^2)/(v_1^2 - 2v_1v_2 - v_2^2), \label{valt}
\end{equation} and with this value of $t$,   it is readily verified that both $\phi(t, u_1, u_2, v_1, v_2)$ and $\phi(t, u_1, u_2, -v_1, v_2)$ become squares, that is, both the remaining sums $a_3+b_4$ and $a_4+b_3$ become perfect squares. 

Thus, on substituting the values of $m, p, q, r, s$ and $t$ given by \eqref{valmpq} and \eqref{valt}, respectively, in the relations \eqref{quadpair1}, we get  two quadruples $\{a_1, a_2, a_3, a_4\}$ and $\{b_1, b_2, b_3, b_4\}$ such that all the 16 sums $a_i+b_j, i=1, \ldots, 4, j=1, \ldots, 4$, are squares. On appropriate scaling, we now get the two quadruples defined by \eqref{quadpairfirst}
 as stated in the theorem.
\end{proof}

As a numerical example, on taking $(u_1, u_2, v_1, v_2)=(3, 1, 4, 3)$, Theorem \ref{thm44first} yields, on appropriate scaling, the  two quadruples,
\[
\begin{aligned}
A & = \{ 44782864, 340218025, 1738222864, 2777290000\}, \\
B & = \{0, 25777136, 1222007600, 1719217136\}
\end{aligned}.
\]
such that all the 16 elements of the sumset $A+ B$ are perfect squares.

\subsection{More examples of pairs of quadruples}\label{twomorepairs}
We will now apply Lemma \ref{onelem} to prove the following theorem which gives two new pairs of quadruples $A$ and $B$ such that all 16 elements of the sumset $A+ B$ are squares. 
\begin{theorem}\label{twoquadruplepairs} If $\{a_1, a_2, a_3, a_4\}$ and $\{b_1, b_2, b_3, b_4\}$ are two quadruples, defined in terms of arbitrary parameters $m, q$ and $r$ by
\begin{equation}
\begin{aligned}
a_1  &  = 16m^2(m^4-1)^2((5m^4 - 2m^2 + 1)q^2 - 2(m^4 -6m^2+ 1)qr \\
& \quad \quad  + (m^4 - 2m^2 + 5)r^2)^2,\\
a_2  &  =  16m^2(m^4 - 1)^2((5m^4 - 2m^2 + 1)q^2 - (m^4 - 2m^2 + 5)r^2)^2,\\
a_3   &  =  16m^2(m^4 - 1)^2((7m^4 - 2m^2 - 1)q^2 - (2m^4 - 12m^2 + 2)qr \\
& \quad \quad - (m^4 + 2m^2 - 7)r^2)^2,\\
a_4   &  =  (m^2 - 1)^2((3m^8 + 40m^6 - 26m^4 - 1)q^2 + 2(m^4 - 6m^2 + 1)^2qr\\
  & \quad \quad - (m^8 + 26m^4 - 40m^2 - 3)r^2)^2,\\
b_1  & =0, \\
b_2  &  = -4m^2((m^4 + 6m^2 - 3)q + (m^4 - 2m^2 + 5)r)\\
 & \quad \quad \times ((3m^4 - 6m^2 - 1)q - (m^4 - 2m^2 + 5)r)\\
 & \quad \quad \times ((5m^4 - 2m^2 + 1)q - (3m^4 - 6m^2 - 1)r)\\
 & \quad \quad \times ((5m^4 - 2m^2 + 1)q + (m^4 + 6m^2 - 3)r),\\
b_3  &  =  -128m^2(m^4 - 1)^3(q^2 - r^2)(3m^2q - m^2r - q + 3r)(m^2q + r),\\
b_4  &  =  16(m^2 + 1)^2(m^2 - 1)^3(m^2q^2 - r^2)\\
 & \quad \quad \times ((m^4 + 6m^3 - 2m - 1)q - (m^4 + 2m^3 - 6m - 1)r)\\
 & \quad \quad \times ((m^4 - 6m^3 + 2m - 1)q - (m^4 - 2m^3 + 6m - 1)r),
\end{aligned}
\label{quadpairsec}
\end{equation}
or, by
\begin{equation}
\begin{aligned}
a_1   & = 16m^2(m^4 - 1)^2((m^6 - 2m^4 + 5m^2)q^2 + (2m^5 - 12m^3 + 2m)qr + \\
 & \quad \quad (5m^4 - 2m^2 + 1)r^2)^2,\\
a_2   & = 16m^2(m^4 - 1)^2((m^6 - 2m^4 + 5m^2)q^2 - (5m^4 - 2m^2 + 1)r^2)^2,\\
a_3   & = (m^2 - 1)^2((m^{10} + 26m^6 - 40m^4 - 3m^2)q^2  \\
 & \quad \quad +(2m^9 - 24m^7 + 76m^5 - 24m^3 + 2m)qr\\
 & \quad \quad - (3m^8 + 40m^6 - 26m^4 - 1)r^2)^2,\\
a_4   & = 16m^2(m^4 - 1)^2((m^6 + 2m^4 - 7m^2)q^2 \\
 & \quad \quad - (2m^5 - 12m^3 + 2m)qr - (7m^4 - 2m^2 - 1)r^2)^2,\\
b_1 & =0, \\
b_2 & =  -4m^2((m^5 + 6m^3 - 3m)q - (5m^4 - 2m^2 + 1)r)\\
  & \quad \quad \times ((3m^5 - 6m^3 - m)q + (5m^4 - 2m^2 + 1)r)\\
  & \quad \quad \times ((m^5 - 2m^3 + 5m)q + (3m^4 - 6m^2 - 1)r)\\
	  & \quad \quad \times ((m^5 - 2m^3 + 5m)q - (m^4 + 6m^2 - 3)r),\\
b_3   & = -16m^2(m^2 + 1)^2(m^2 - 1)^3(q^2 - r^2)\\
& \quad \quad \times ((m^5 - 2m^4 + 6m^2 - m)q + (m^4 - 6m^3 + 2m - 1)r)\\
  & \quad \quad \times((m^5 + 2m^4 - 6m^2 - m)q + (m^4 + 6m^3 - 2m - 1)r),\\
b_4   & = 128m^3(m^4 - 1)^3(mr - q)(m^2q^2 - r^2)((m^3 - 3m)q + (3m^2 - 1)r),
\end{aligned}
\label{quadpairthird}
\end{equation}
then for each of these two pairs of quadruples, all the 16 sums $a_i+b_j, i=1, \ldots, 4, j=1, \ldots, 4$, are squares.
\end{theorem}
\begin{proof}
In Section \ref{quadruplesfirstpair} we have already obtained two matrices $E$ and $G$ that satisfy the relations \eqref{releg}. We will now construct a matrix $H$ such that the relations \eqref{releh} and \eqref{relgh} are satisfied, and we will then apply Lemma \ref{onelem} to obtain two quadruples $\{a_1, a_2, a_3, a_4\}$ and $\{b_1, b_2, b_3, b_4\}$ such that all the 16 sums $a_i+b_j, i=1, \ldots, 4, j=1, \ldots, 4$, are squares. 

We will construct the matrix $H$ by replacing $p, q, r, s$, in Euler's matrix $M$ by 
\[
\begin{aligned}
((n+1)p - (n-1)s)/2, \quad & ((n+1)q - (n-1)r)/2, \\
  ((-n+1)q + (n+1)r)/2,\quad   & ((-n+1)p + (n+1)s)/2, 
	\end{aligned}
\] 
respectively, where $n$ is an arbitrary parameter, and dividing the entries of the resulting matrix by $n^2$, when we get the following semi-magic matrix:
\begin{equation}
H=\begin{bmatrix} (p^2 + q^2 - r^2 - s^2)^2 & h_{12}^2    & (2qs -2pr)^2   \\
h_{21}^2 & (p^2 - q^2 + r^2 - s^2)^2  & h_{23}^2   \\
h_{31}^2 &  (2rs - 2pq)^2 & h_{33}^2
\end{bmatrix}
\label{matH}
\end{equation}
where
\begin{equation}
\begin{aligned}
h_{12} & =  ((p^2 - 2ps + q^2 - 2qr + r^2 + s^2)n^2 \\
& \quad \quad - p^2 - 2ps - q^2 - 2qr - r^2 - s^2)/(2n)    \\,
h_{21} & =  ((p - q + r - s)(p + q - r - s)n^2 \\
& \quad \quad  - (p + q + r + s)(p - q - r + s))/(2n)      \\,
h_{23} & = ((q - r)(p - s)n^2 + (q + r)(p + s))/n, \\
h_{31} & =  ((q - r)(p - s)n^2 - (q + r)(p + s))/n, \\
h_{33} & = ((p - q + r - s)(p + q - r - s)n^2 \\
& \quad \quad  + (p + q + r + s)(p - q - r + s))/(2n).
\end{aligned}
\label{valhij}
\end{equation}
It is readily seen that the entries of the matrices $E$ and $H$ satisfy the relations \eqref{releh}. 

We now impose the conditions \eqref{relgh}, that is, $g_{11}^2=h_{12}^2$ and $g_{22}^2=h_{31}^2$ where the values of $g_{11}, g_{22}$ and $h_{12}, h_{31}$ are defined by \eqref{valgij} and \eqref{valhij}, respectively. These conditions will be satisfied if $g_{11}=\pm h_{12}$ and $g_{22}=\pm h_{31}$, that is, we get four possible pairs of equations, and each of these pairs of equations can be solved to get a matrix $H$ that satisfies all the conditions of Lemma \ref{onelem}, and we can thus  obtain the desired pairs of quadruples.

The simplest pairs of quadruples are obtained by solving the pair of equations $g_{11}=- h_{12}$ and $g_{22}=- h_{31}$, which may be written as 
\begin{equation}
\begin{aligned}
 (m^2 p^2+m^2 q^2-r^2-s^2)/m & = -\{(p^2 - 2ps + q^2 - 2qr + r^2 + s^2)n^2 \\
 & \quad \quad - p^2 - 2ps - q^2 - 2qr - r^2 - s^2\}/(2n),
\end{aligned}
\label{eqgh1}
\end{equation}
and
\begin{equation}
\begin{aligned}
 (m^2 p^2-m^2 q^2+r^2-s^2)/m & = -\{(q - r)(p - s)n^2 - (q + r)(p + s)\}/n.
\end{aligned}
\label{eqgh2}
\end{equation}   

On taking the differences of the respective sides of Eqs. \eqref{eqgh1} and \eqref{eqgh2}, we get
\begin{multline}
2(mq - r)(mq + r)/m = -((n(p - q + r - s) + p - q - r + s)\\
\times (n(p - q + r - s) - p + q + r - s))/(2n). \label{diffgh12}
\end{multline}
Thus there exists a rational number $t$ such that
\begin{equation}
2n(mq-r)=-t(n(p - q + r - s) + p - q - r + s)m, \label{difft1} 
\end{equation}
and  
\begin{equation}
2t(mq+r)=n(p - q + r -s) - p + q + r - s.
\label{difft2}
\end{equation}
Eqs. \eqref{difft1} and \eqref{difft2} may be considered as two linear equations in the variables $p$ and $s$, and on solving them, we get,
\begin{equation}
\begin{aligned}
p & = -(m(n - 1)(mq + r)t^2 - 2qmnt + n(n + 1)(mq - r))/(2mnt),\\
s & = -(m(n + 1)(mq + r)t^2 - 2rmnt + n(n - 1)(mq - r))/(2mnt).
\label{valps}
\end{aligned}
\end{equation}
On substituting the values of $p$ and $s$ given by \eqref{valps} in Eq. \eqref{eqgh1}, transposing all terms to the left-hand side, and factorizing, we get
\begin{multline}
\{(m - 1)n^2 + (m^2t^2 + mt^2 - 4mt + m + 1)n - mt^2(m - 1)\}\\
\times \{(m + 1)(mq - r)^2n^2+ (m - 1)(m^3q^2t^2 + 2m^2qrt^2 + mr^2t^2 + m^2q^2 - 2mqr + r^2)n\\
 - mt^2(m + 1)(mq + r)^2\} =0. \label{eqgh1a}
\end{multline}
To solve Eq. \eqref{eqgh1a}, we simply equate the first factor to 0, that is, we will solve the equation,
\begin{equation}
(m - 1)n^2 + (m^2t^2 + mt^2 - 4mt + m + 1)n - mt^2(m - 1)=0. \label{eqgh1b}
\end{equation}

Now, Eq. \eqref{eqgh1b} will have a rational solution for $n$ if its discriminant $d$ with respect to $n$ is a perfect square, where
\begin{equation}
d =  (m + 1)^2m^2t^4 - 8m^2(m + 1)t^3 + 6(m + 1)^2mt^2 - 8m(m + 1)t + (m + 1)^2, \label{vald}
\end{equation}
Since $d$ is a quartic function of $t$ and the coefficient of $t^4$ as well as the constant term are both $0$, we can, following a method Fermat (described by Dickson \cite[p. 639]{Di}), readily find a value of $t$ that makes  $d$ a perfect square, and we thus get
 $t=2(m + 1)/(m^2 + 1)$. With this value of $t$, the discriminant $d$ becomes a perfect square, and Eq. \eqref{eqgh1b} can be solved to get two rational roots for $n$. We thus get the following two solutions of Eq. \eqref{eqgh1b}:
\begin{equation}
(t, n)= (2(m + 1)/(m^2 + 1), 4m(m^2 - 1)/(m^2  + 1)^2, \label{sol1tn} 
\end{equation}
 and 
\begin{equation}
 (t, n)= (2(m + 1)/(m^2 + 1),-(m + 1)/(m - 1)). \label{sol2tn}
\end{equation}

Substituting the values of $t$ and $n$ given by \eqref{sol1tn} in the relations \eqref{valps}, we get
\begin{equation}
\begin{aligned}
p & = -((m^4 - 6m^2 + 1)q - (m^4 - 2m^2 + 5)r)/(2(m^4 - 1)), \\
s & = -((5m^4 - 2m^2 + 1)q - (m^4 - 6m^2 + 1)r)/(2(m^4 - 1)).
\end{aligned}
\label{valpsfin1}
\end{equation}
With the values of $p$ and $s$ given by \eqref{valpsfin1}, Eqs. \eqref{eqgh1} and \eqref{eqgh2} are satisfied, and hence we get three matrices $E= M$ (defined by \eqref{Euler}), $G$ and $H$ defined by \eqref{matG} and \eqref{matH}, respectively, satisfying all the conditions of Lemma \ref{onelem} and, on using the relations \eqref{defAB} and \eqref{vala4b4},  we  get two quadruples $\{a_1, a_2, a_3, a_4\}$ and $\{b_1, b_2, b_3, b_4\}$  such that all the 16 sums $a_i+b_j, i=1, \ldots, 4, j=1, \ldots, 4$, are squares. On appropriate scaling, we get the pair of quadruples \eqref{quadpairsec} given  by the theorem.

Similarly,  we can use  the solution \eqref{sol2tn} of Eq. \eqref{eqgh1b} to obtain a new set of three matrices $F, G, $ and $H$ satisfying the conditions of Lemma \ref{onelem} and we thus obtain the second pair of quadruples \eqref{quadpairthird} given  by the theorem.
\end{proof}

As a numerical application of Theorem \ref{twoquadruplepairs}, on taking $(m, q, r)=(2, 1, 3)$ in the relations \eqref{quadpairsec}, we get two quadruples which, on appropriate scaling, yield the two quadruples,
\[
\begin{aligned}
A & =\{14400, 266256, 435600, 12110400\},\\
B & = \{0, 104625, 223744, 12096000\},
\end{aligned}
\]
such that all the 16 elements of the sumset $A+ B$ are perfect squares.

While Theorem \ref{twoquadruplepairs} gives two pairs of quadruples, in parametric terms,  with the desired property, many more such examples may be obtained using Lemma 2. For instance, we could solve Eq. \eqref{eqgh1a} by equating to 0 the second factor on the left-hand side of \eqref{eqgh1a} to get the matrix $H$. Further, as already mentioned, the conditions \eqref{relgh} may be satisfied by solving any of  four possible pairs of equations arising from the conditions $g_{11}= \pm h_{12}$ and $g_{22}= \pm h_{31}$. In Theorem \ref{twoquadruplepairs} we solved just one of these pairs of equations. Solving the remaining three pairs of equations also yields the desired matrix $H$ and the pairs of quadruples. These solutions are, however, a little more cumbersome than the examples given by the theorem.

\section{Some open problems}
It would be interesting to find a quintuple $\{a_1, a_2,  \ldots, a_5\}$ and a quadruple $\{b_1, b_2,  \ldots, b_4\}$ such that all the 20 sums $a_i+b_j, i=1, \ldots, 5, j=1, \ldots, 4$, are squares. Since we have obtained multi-parameter solutions for  a quintuple and a triple as well as for a pair of quadruples with the desired property, it seems quite likely that there exist infinitely many examples of a quintuple $\{a_1, \ldots, a_5\}$ and a quadruple $\{b_1, \ldots, b_4\}$ such that all the 20 sums 
$a_i+b_j$ are squares. It would be much more challenging to find two quintuples $\{a_1, \ldots, a_5\}$ and  $\{b_1, \ldots, b_5\}$ such that all the 25 sums $a_i+b_j$ are squares. We leave these problems for future investigations.

\noindent Postal Address: 13/4 A Clay Square, \\
\hspace{1.1in} Lucknow - 226001,\\
\hspace{1.1in}  INDIA \\
\medskip
\noindent e-mail address: ajaic203@yahoo.com

\end{document}